\documentclass[11pt,a4paper]{article}
\pdfoutput=1 
\usepackage{bbm}
\usepackage{epsf,epsfig,amsfonts,amsgen,amsmath,amstext,amsbsy,amsopn,amsthm,cases,listings,color
}
\usepackage{ebezier,eepic}
\usepackage{color}
\usepackage{multirow}
\usepackage{epstopdf}
\usepackage{graphicx}
\usepackage{pgf,tikz}
\usepackage{mathrsfs}
\usepackage[marginal]{footmisc}
\usepackage{enumitem}
\usepackage[titletoc]{appendix}
\usepackage{booktabs}
\usepackage{url}
\usepackage{mathtools}

\usepackage{pgfplots}
\usepackage{authblk}
\usepackage{amssymb}

\usepackage{wasysym}

\usepackage{empheq}

\usepackage{dsfont}

\usepackage{caption}
\usepackage{subcaption}
\pgfplotsset{compat=1.18}
\usepackage{mathrsfs}
\usepackage{wasysym} 
\usetikzlibrary{arrows}
\usepackage{aligned-overset}
\usepackage{bm}

\usepackage[normalem]{ulem}
\usepackage[backref=page]{hyperref}
\usepackage{makecell}

\allowdisplaybreaks[1]

\definecolor{uuuuuu}{rgb}{0.27,0.27,0.27}
\definecolor{sqsqsq}{rgb}{0.1255,0.1255,0.1255}

\newtheorem{definition}{Definition}
\newtheorem{theorem}[definition]{Theorem}
\newtheorem{lemma}[definition]{Lemma}

\newtheorem{fact}[definition]{Fact}
\newenvironment{pf}{\noindent{\bf Proof}}

\theoremstyle{remark}
\newtheorem{remark}[definition]{Remark}

\newcommand{\C}[1]{\mathcal{#1}}
\newcommand{\I}[1]{{\mathbbm #1}}

\newcommand{\hide}[1]{}
\renewcommand{\mid}{:}

\begin{document}
\title{\bf\Large A note on the minimum size of Tur\'{a}n systems}
\author{Xizhi Liu}
\author{Oleg Pikhurko}
\affil{Mathematics Institute and DIMAP,
             University of Warwick,
             Coventry, 
             UK
}
\date{\today}
\maketitle
\begin{abstract}
For positive integers $n \ge s > r$, a \emph{Tur\'{a}n $(n,s,r)$-system} is an $n$-vertex $r$-graph in which every set of $s$ vertices contains at least one edge.
Let $T(n,s,r)$ denote the the minimum size of a Tur\'{a}n $(n,s,r)$-system. 

Upper bounds on $T(n,s,r)$ were established by Sidorenko~\cite{Sid97} for the case $s-r = \Omega(r/\ln r)$ (based on a construction of Frankl--R\"{o}dl~\cite{FR85}) and by a number of authors in the case $s-r = O(1)$. 
In this note, we establish upper bounds in the remaining range $O(1)<s-r = O(r/\ln r)$.
\end{abstract}
\section{Introduction}
Given an integer $r\ge 2$, an \textbf{$r$-uniform hypergraph} (henceforth an \textbf{$r$-graph}) $\mathcal{H}$ is a collection of $r$-subsets of some set $V$. We call $V$ the \textbf{vertex set} of $\mathcal{H}$ and denote it by $V(\mathcal{H})$. When $V$ is understood, we usually identify a hypergraph $\mathcal{H}$ with its set of edges.

For positive integers $n \ge s > r$, a \textbf{Tur\'{a}n $(n,s,r)$-system} is an $r$-graph $\mathcal{H}$ on an $n$-set $V$ such that every $s$-subset $S\subseteq V$ contains at least one edge from $\mathcal{H}$. 
Denote by $T(n,s,r)$ the smallest \textbf{size} (i.e.\ the number of edges) of a Tur\'{a}n $(n,s,r)$-system. 
Observe that $T(n,s,r) = \binom{n}{r} - \mathrm{ex}(n,K_{s}^{r})$, where $\mathrm{ex}(n,K_{s}^{r})$ denotes the Tur\'{a}n number of the complete $r$-graph on $s$ vertices $K_{s}^{r}$. 
A simple averaging argument shows that $\mathrm{ex}(n,K_{s}^{r})/\binom{n}{r}$ is non-increasing (see e.g.~\cite{KNS64}), and hence the following limit exists: 
\begin{align}\label{eq:density}
    t(s,r)
    \coloneqq \lim_{n\to \infty}\frac{T(n,s,r)}{\binom{n}{r}}. 
\end{align}
Determining the value of $t(s,r)$ is a central topic in Extremal Combinatorics. 
The seminal Tur\'{a}n Theorem~\cite{Tur41} established that $t(s,2) = \frac{1}{s-1}$ for all $s \ge 3$ (with the case $s=3$ solved earlier by Mantel~\cite{Mantel07}). 
However, the exact value of $t(s,r)$ remains unknown for any pair $(s,t)$ satisfying $s > r \ge 3$, despite decades of active attempts. 
Erd\H{o}s~\cite{Erd81} offered \$500 for the determination of $t(s,r)$ for any $s > r \ge 3$. 
Tur\'{a}n and other researchers conjectured that $t(s,3) = \frac{4}{(s-1)^2}$ for $s \ge 4$. 
Various constructions achieving this bound are known (see e.g.~\cite{Sid95}). 
For $r \ge 4$,  there is no general conjectured value for $t(s,r)$, except for the case $(s,r) = (5,4)$ (see~\cite{Gir90,Mar09}). 
For further related results, we refer the reader to surveys such as~\cite{Caen94Survey,Fur91,Sid95,Kee11}. 

In this note, we focus on the case where $r \to \infty$, and all asymptotics are taken with respect to $r$ unless otherwise specified. The trivial lower bound is $t(s,r) \ge {1}/{\binom{s}{r}}$, which follows e.g.\ from the monotonicity of the ratio in~\eqref{eq:density}. For convenience, let us define
\begin{align*}
    \mu(s,r)
    \coloneqq t(s,r) \cdot \binom{s}{r}.
\end{align*}
 Thus $\mu(s,r)$ is always at least 1.
 
The best-known general lower bound, $t(s,r) \ge 1/{\binom{s-1}{r-1}}$ (i.e.\ $\mu(s,r)\ge s/r$) is due to de Caen~\cite{Caen83}. 
In particular, $t(r+1, r) \ge \frac{1}{r}$, a result that was independently proved by de Caen~\cite{Decaen83ac}, Sidorenko~\cite{Sidorenko82}, and Tazawa and Shirakura~\cite{TazawaShirakura83}.
Further improvements on $t(r+1, r)$ in lower order terms were made by Giraud (unpublished), Chung--Lu~\cite{CL99}, and Lu--Zhao~\cite{LZ09}. 

Improving previous upper bounds established by Sidorenko~\cite{Sid81}, Kim--Roush~\cite{KR83}, Frankl--R\"{o}dl~\cite{FR85}, and Sidorenko~\cite{Sid97}, the second author established the following upper bound, which disproved the conjecture of de Caen~\cite{Caen94Survey} that $r \cdot t(r+1, r) \to \infty$. 

\begin{theorem}[\cite{Pik24}]\label{THM:Pikhurko}
    For every integer $R \ge 1$, it holds that 
    \begin{align*}
        \mu(r+R,r)
        \le \alpha + o(1), 
        \quad\text{as}\ r \to \infty, 
    \end{align*}
    where $\alpha \coloneqq {(c_0+1)^{R+1}}/{c_0^{R}}$ with $c_0 = c_0(R)$ being the largest real root of the equation $\mathrm{e}^{x} = (x+1)^{R+1}$. 
    In particular, $
    \mu(r+1,r)\le 4.911$ for all sufficiently large $r$. 
\end{theorem}
An immediate corollary of Theorem~\ref{THM:Pikhurko} (for derivation see~\cite[Corollary~1.3]{Pik24}) is that, for all sufficiently large $R$,  
\begin{align}\label{equ:upper-bound-Pikhurko}
    \limsup_{r\to \infty}~\mu(r+R,r)
    \le R \ln R + 3R\ln\ln R=(1+o(1)) R\ln R. 
\end{align}
This improves asymptotically the previous bound by Frankl--R\"{o}dl~\cite{FR85} which states that, for any fixed $R \ge 1$, we have
\begin{align*}
    \mu(r+R,r)
    \le (1+o_{R}(1)) R (R+4) \ln r,\quad \mbox{as $r\to\infty$}. 
\end{align*}
For the case $R \ge \frac{r}{\log_{2} r}$, Sidorenko~\cite{Sid97} (by analyzing the construction of Frankl--R\"{o}dl~\cite{FR85} in this regime) proved that
    \begin{align}\label{equ:upper-bound-Sidorenko-a}
        \mu(r+R, r) 
        \le (1+o(1)) R\ln \binom{r+R}{R}. 
    \end{align}

%

Although Tur\'an systems were actively studied, it seems that no general upper bounds on $t(r+R, r)$ have been published in  the intermediate regime $1 \ll R \le r/\log_2 r$. This is the case we address in this note.  In brief, we show that Sidorenko's bound in~\eqref{equ:upper-bound-Sidorenko-a} applies in the whole range as along as $R$ is sufficiently large, while  the asymptotic bound in~\eqref{equ:upper-bound-Pikhurko} can be extended from constant $R$ to any $R=o(\sqrt r)$.

\begin{theorem}\label{THM:main-mu-r-R-general}
    For every $\varepsilon > 0$, there exists $r_0$ such that the following statements hold for all $r, R$ satisfying $R \ge r_0$. 
    \begin{enumerate}[label=(\roman*)]
        \item\label{THM:main-mu-r-R-general-1} It holds that $\mu(r+R, r)
           \le (1+\varepsilon) R \ln \binom{r+R}{R}$.

      \item\label{THM:main-mu-r-R-general-2} Suppose that $R \le \sqrt{18 r \ln r}$. Then 
        \begin{align*}
             \mu(r+R, r)
            \le \mathrm{e}^{18 R^2/r} \cdot (1+\varepsilon) R \ln R. 
        \end{align*}
    \end{enumerate}
\end{theorem}

\section{Proofs}\label{SEC:proof}
In this section, we prove Theorem~\ref{THM:main-mu-r-R-general}. We will use the following notation.
For an integer $n\ge 1$ and a set $X$, we denote $[n]:=\{1,\dots,n\}$ and $\binom{X}{n}:=\{Y\subseteq X\mid |Y|=n\}$. 

Let us begin with the proof of Theorem~\ref{THM:main-mu-r-R-general}~\ref{THM:main-mu-r-R-general-1}. 
\begin{proof}[Proof of Theorem~\ref{THM:main-mu-r-R-general}~\ref{THM:main-mu-r-R-general-1}]
    Given $\varepsilon >0$, let $r_0$ be sufficiently large. Take any $r,R$ with $R \ge r_0$. By~\eqref{equ:upper-bound-Sidorenko-a}, we can assume that e.g.\ $R \le \frac{r}{\ln r}$. Let $s \coloneqq r+R$. Define 
    \begin{align*}
        N
        \coloneqq \left\lfloor r(r-1){s\choose R}/(2R)\right\rfloor 
        \quad\text{and}\quad 
        \ell
        \coloneqq \left\lfloor {s\choose R}/\ln\left({s\choose R}^2{N-s\choose R}\right)\right\rfloor.
    \end{align*}    
    In the rest of the proof, we repeat the construction of Frankl and R\"odl~{\cite[Theorem 3]{FR85}}, arguing that our choices of $N$ and $\ell$ give the stated bound.
 
    Consider a random colouring $c : {[N]\choose r} \to [\ell]$. For each $s$-set $A\in {[N]\choose s}$ we have a bad event that some colour is not present in ${A\choose r}$. Its probability $p$ is at most $\ell (1-1/\ell)^{{s\choose r}}\le \ell \mathrm{e}^{-{s\choose R}/\ell}$ (note that $\binom{s}{r} = \binom{s}{s-r} = \binom{s}{R}$). Define the obvious dependency graph $D$ on ${[N]\choose s}$ where $A\sim B$ if $|A\cap B|\ge r$ (including loops). It is regular of degree $\Delta:=\sum_{i=r}^s {s\choose i}{N-s\choose s-i}$. 
 
We want to apply the Lov\'asz Local Lemma~\cite{EL75} (see also e.g.~\cite[Corollary~5.1.2]{AlonSpencer16pm}), so we have to check that $\mathrm{e} p\Delta<1$, for which it suffices to prove 
 \begin{align}\label{eq:LLL}
     \mathrm{e}^{{s\choose R}/\ell} 
     >\mathrm{e} \ell \Delta.
 \end{align}
 First, we estimate $\Delta$. Consider the ratio of two consecutive terms:
 $$
 \frac{{s\choose i+1}{N-s\choose s-i-1}}{{s\choose i}{N-s\choose s-i}} = \frac{(s-i)^2}{(i+1)(N-2s+i+1)}.
 $$
 Since $3\le R\le s/2$ and $N\ge {s\choose R}\ge {s\choose 3}$, this ratio is at most say $1/2$ for every $i\in [r,s]$ and thus we can bound
 \begin{align}\label{eq:Delta}
     \Delta \le 2{s\choose r}{N-s\choose s-r}=2{s\choose R}{N-s\choose R}.
 \end{align}
 Thus and by the choice of $\ell$ we have 
 $$
 \mathrm{e}^{{s\choose r}/\ell}\ge {s\choose R}^2{N-s\choose R} \ge \mathrm{e}\cdot  \frac{{s\choose R}}{2\ln{s\choose R}} \cdot 2{s\choose R}{N-s\choose R}\ge \mathrm{e} \ell \Delta,
 $$
 as desired. (In fact, our definition of  $\ell$ comes from~\eqref{eq:LLL}, given $N$.)
 
 Thus the Lov\'asz Local Lemma applies and there is a colouring $c$ with no bad events, meaning that each colour gives a Tur\'an $(N,s,r)$-system on $[N]$. Let $\C A\subseteq {[N]\choose r}$ be the $r$-graph formed by the least frequent colour. We have
 \begin{align}\label{eq:A}
     |\C A|\le \frac{1}{\ell}{N\choose r}.
 \end{align}
Now let $n:=mN$ with integer $m\to\infty$. Our Tur\'an $(n,s,r)$-system $\C B$ on $[n]$ is made of a blowup of $\C A$ plus all $r$-sets that intersect at least one part in more than one vertex. We have
 $$
 |\C B|\le m^r|\C A|+N{m\choose 2}{mN-2\choose r-2}\le m^r\cdot \frac{1}{\ell}{N\choose r} + \frac{r(r-1)}{2N} {mN\choose r} \le {mN\choose r} f,
 $$
 where $f:=\frac1{\ell}+\frac{r(r-1)}{2N}$. Since $f$ does not depend on $m$, it gives an upper bound on $t(r+R,r)$ and thus $f\cdot {s\choose R}\ge \mu(s,r)$. Note that, rather roughly, we have $N\le {s\choose R}^2$ and $\ln N\le 2R\ln s$. Therefore,
 $$
 \ell\ge  \frac{{s\choose R}}{2\ln{s\choose R}+R\ln N}-1
 $$ 
 can be forced to be arbitrarily large if $r_0$ was sufficiently large.
 Thus the rounding in the definition of $\ell$ gives only a multiplicative term that is arbitrarily close to $1$
 and we have
\begin{align*}  
\binom sR\, f&\le(1+\varepsilon /4)\left(2\ln \binom sR+R\ln N + \frac {r(r-1){s\choose R}}{2N}\right)\\
 &\le(1+\varepsilon /2)\left(2\ln \binom sR + R\left(2\ln r+\ln \binom sR\right) + R\right) \le (1+\varepsilon ) R\ln \binom sR,
\end{align*}
 proving Theorem~\ref{THM:main-mu-r-R-general}~\ref{THM:main-mu-r-R-general-1}
  (In fact, our choice of $N$ was given by the fact that $\ln x+c/x$ with fixed $c>0$ is minimized on $(\mathrm{e},\infty)$ for $x=c$.)    
\end{proof}


Next, we prove Theorem~\ref{THM:main-mu-r-R-general}~\ref{THM:main-mu-r-R-general-2}.
Before doing so, let us present some preliminary results. 

The following lemma is derived from the proof of~{\cite[Lemma~2.3]{Pik24}}.
\begin{lemma}[\cite{Pik24}]\label{LEMMA:recursion}
    For all integers $r, R \ge 1$, $k \in [R, r-1]$, and a real $c \in \left[0, \binom{k}{R}\right]$, it holds that 
    \begin{align}\label{equ:mu-recursion}
        \mu(r+R,r)
        & \le \binom{r+R}{R} \left(\frac{c}{\binom{k}{R}} + \frac{\mu(r-k+R, r-k)}{\mathrm{e}^{c} \cdot \binom{r-k+R}{R}} \right).
    \end{align}
\end{lemma}

\begin{proof}[Sketch of Proof] Let $S$ be a random subset of $\binom{[n]}{k-R}$ where each $(k-R)$-set is included into $S$ with probability $p:=c/\binom{k}{R}$. Let $S^*\subseteq \binom{[n]}{r}$ consist of those $r$-sets $\{x_1<\dots<x_r\}$ such that $\{x_1,\dots,x_{k-R}\}\in S$, that is, we include an $r$-set into $S^*$ if its initial $(k-R)$-segment is in $S$. Let $T\subseteq \binom{[n]}{k}$ consist of those $k$-sets $X$ such that $\binom{X}{k-R}\cap S=\emptyset$, that is, $X$ is not hit by any set in $S$. For each $y\in [n]$, take
a minimum Tur\'an $(n-y,r-k+R,r-k)$-system $F_y$ on $\{y+1,\dots,n\}$. Let $T^*\subseteq \binom{[n]}r$ be be union over $Y\in T$ of the $r$-graphs
$\{Y\cup Z\mid Z\in F_{\max Y}\}$. Informally speaking, we extend every $Y\in T$ by a minimum Tur\'an system to the right of $Y$. 

It is easy to check that $G:=S^*\cup T^*$ is  a Tur\'an $(n,r+R,r)$-system, regardless of the choice of $S$. By taking $S$ such that $|G|$ is at most its expected value, it is routine to see that
 \begin{eqnarray*}
 T(n,r+R,r)&\le& \I E|S^*|+\I E|T^*|
\\&=& p {n\choose r}
+\sum_{y=k}^{n}  
\left(1-p
\right)^{{k\choose R}}
{y-1\choose k-1} \cdot T(n-y,r-k+R,r-k)
\\&\le& \left(\frac{c}{{k\choose R}}+  \frac{\mathrm{e}^{-c}}{{r-k+R\choose R}}\mu(r-k+R,r-k)\right){n\choose r},
 \end{eqnarray*}
 giving the required.
    \end{proof}

\begin{fact}\label{FACT:inequalities}
    For any integers $r_1\ge r_2> R$, we have 
    \begin{align*}
        \binom{r_1}{R}/\binom{r_2}{R}
        & = \prod_{i=0}^{R-1}\frac{r_1-i}{r_2-i}
        \le \left(\frac{r_1-R}{r_2-R}\right)^{R} . 
    \end{align*}
\end{fact}

\hide{
\begin{fact}\label{FACT:inequalities}
    Let $r > k \ge 1$ and $R \ge 1$ be integers. 
    Then
    \begin{align*}
        \binom{r+R}{R}/\binom{k}{R}
        & = \prod_{i=0}^{R-1}\frac{r+R-i}{k-i}
        \le \left(\frac{r}{k-R}\right)^{R} \quad\text{ if $r+R\ge k$,\quad and}\\
        \binom{r+R}{R}/\binom{r-k+R}{R} 
        & = \prod_{i=0}^{R-1}\frac{r+R-i}{r-k+R-i}
        \le \left(\frac{r}{r-k}\right)^{R}. 
    \end{align*}
\end{fact}
}
\begin{fact}\label{FACT:inequalities-b}
    Let $r \ge 1, R \ge 1$ be integers and $\delta$ be a real number satisfying $18 R^2/r \le \delta \le R$.
    Let $k \coloneqq \left\lceil \frac{R r}{R+\delta} \right\rceil + R$. 
    Then 
    \begin{align*}
        k \le r-1, \quad 
        \frac{r}{k-R} \le 1 + \frac{\delta}{R}, \quad\text{and}\quad 
        \frac{r}{r-k} 
        \le \frac{3R}{\delta}. 
    \end{align*}
\end{fact}
\begin{proof}[Proof of Fact~\ref{FACT:inequalities-b}]
    Since $\delta \ge \frac{18 R^2}{r}$, straightforward calculations show that 
    \begin{align*}
        \frac{\delta r}{R+\delta} - R - 2
        & = \frac{\delta r - (R+\delta)(R+2)}{R+\delta}
        \ge \frac{18 R^2 - (R+R)(R+2)}{R+\delta}
        \ge 0 \quad\text{and}\quad\\
        \frac{R r}{R + \delta/2} - \frac{R r}{R + \delta}
        & = \frac{\delta  r R}{(R + \delta) (2 R + \delta)} 
        \ge \frac{18 R^3}{(R+R)(2R+R)}
        \ge R+1. 
    \end{align*}
    It follows that 
    \begin{align*}
        k 
        & \le \frac{R r}{R+\delta} + 1 + R
        = r - 1 - \left(\frac{\delta r}{R+\delta} - R - 2\right)
        \le r-1, \\
        \frac{r}{k-R}
        & \le \frac{r}{\frac{R r}{R + \delta} + R -R}
        = \frac{R+\delta}{R}
        = 1 + \frac{\delta}{R},  \quad\text{and}\\
        \frac{r}{r-k}
        & \le \frac{r}{r- \frac{R r}{R + \delta} -1 - R}
        \le \frac{r}{r- \frac{R r}{R + \delta/2}}
        = \frac{2R+\delta}{\delta}
        \le \frac{3R}{\delta}, 
    \end{align*}
    which proves Fact~\ref{FACT:inequalities-b}. 
\end{proof}

We are now ready to present the proof of Theorem~\ref{THM:main-mu-r-R-general}~\ref{THM:main-mu-r-R-general-2}. 
\begin{proof}[Proof of Theorem~\ref{THM:main-mu-r-R-general}~\ref{THM:main-mu-r-R-general-2}]
    Given $\varepsilon > 0$, choose a sufficiently small real $\varepsilon_1 > 0$ and then a sufficiently large integer~$r_0$. 
    Take any integers $r, R$ such that $R \ge r_0$ and $R \le \sqrt{18 r \ln r}$.\medskip 
    
\noindent\textbf{Case 1.}  Suppose that $R \ge \ln r$. 

We define
    \begin{align*}
        \delta
        \coloneqq \max\left\{\varepsilon_1,~\frac{18 R^2}{r}\right\}, \quad 
        k 
        \coloneqq \left\lceil \frac{R r}{R+\delta} \right\rceil + R, 
        \quad\text{and}\quad 
        c \coloneqq R \ln\left(\frac{3R}{\delta}\right) + \ln(2R^3). 
    \end{align*}

Clearly, $c\le \binom{k}{R}$.   Since $R$ is large and $\ln r \le R$, it follows from
    Theorem~\ref{THM:main-mu-r-R-general}~\ref{THM:main-mu-r-R-general-1} that, for example,
    \begin{align*}
        \mu(r-k+R, r-k)
        & \le 2\, R \ln \binom{r-k+R}{R} 
         \le 2\, R^2 \ln r
        \le 2R^3. 
    \end{align*}
    Combining this with Lemma~\ref{LEMMA:recursion}, Facts~\ref{FACT:inequalities} and~\ref{FACT:inequalities-b}, we obtain 
    \begin{align*}
        \mu(r+R,r)
        & \le \binom{r+R}{R} \left(\frac{c}{\binom{k}{R}} + \frac{2R^3}{\mathrm{e}^{c} \cdot \binom{r-k+R}{R}}\right) \\
        & \le \left(\frac{r}{k-R}\right)^{R} \cdot c + \left(\frac{r}{r-k}\right)^{R} \cdot \frac{2R^3}{\mathrm{e}^c} \\
        & \le \left(1+\frac{\delta}{R}\right)^{R} \cdot c + \left(\frac{3R}{\delta}\right)^{R} \cdot \frac{2R^3}{\mathrm{e}^c} 
        \le \mathrm{e}^{\delta} \cdot c + \mathrm{e}^{R \ln \left(\frac{3R}{\delta}\right) - c} \cdot 2R^3
        \le \mathrm{e}^{\delta} \cdot c + 1. 
    \end{align*}
    If $\frac{18 R^2}{r}\le \varepsilon_1$, that is, $R \le \sqrt{\varepsilon_1 r/18}$, then 
    \begin{align*}
        \mu(r+R,r)
        &\le \mathrm{e}^{\varepsilon_1} \cdot c + 1 
         \le (1+ 2\varepsilon_1) \left(R \ln\left(\frac{3R}{\varepsilon_1}\right) + \ln(2R^3)\right) + 1 \\
        & \le (1+ 2\varepsilon_1) \left(R \left(\ln R + \ln\left(\frac{3}{\varepsilon_1}\right)\right) + 3\ln R + \ln 2\right) + 1 \\
        & \le (1+ \varepsilon) R \ln R, 
    \end{align*}
    as desired. 

    If $\frac{18 R^2}{r}>\varepsilon_1$, that is, $r \le \frac{18 R^2}{\varepsilon_1}$, then 
    \begin{align*}
        \mu(r+R,r)
        &\le \mathrm{e}^{18R^2/r} \cdot c + 1 
         \le \mathrm{e}^{18R^2/r} \left(R \ln\left(\frac{r}{6R}\right) + \ln(2R^3)\right) + 1 \\
        & \le \mathrm{e}^{18R^2/r} \left(R \ln\left(\frac{3R}{\varepsilon_1}\right) + \ln(2R^3)\right) + 1 \\
        & \le \mathrm{e}^{18R^2/r} (1+\varepsilon) R \ln R,
    \end{align*}
    also as desired.\medskip

\noindent\textbf{Case 2.}  Suppose that $R < \ln r$. 

Let $r_1 \coloneqq r$ and, inductively for $i=1,2,\dots$,  define 
   \begin{align*}
        k_i \coloneqq \left\lceil \frac{R r_i}{R+\varepsilon_1} \right\rceil + R, 
        \quad\text{and}\quad  
        r_{i+1} \coloneqq r_i - k_i;
    \end{align*}
    if $r_{i+1}< 18R^2/\varepsilon_1$ then
 let $t:=i$ and stop. Since $r_i$ decreases each time, this process terminates.
  
    We prove by backward induction on $i\in [t]$ that  
    \begin{align}\label{equ:mu-ri-induction}
        \mu(r_i+R, r_i) \le (1+\varepsilon) R\ln R. 
    \end{align}

First, consider the base case $i = t$.   
    %
    Note that, for $i\le t$, we have  by $r_i\ge 18R^2/\varepsilon_1$ that
    \begin{align*}
        r_{i+1}
        = r_i - k_i 
        \ge r_i - \frac{R r_i}{R+\varepsilon_1} - 1 - R
        \ge \frac{\varepsilon_1 r_i}{R+\varepsilon_1} -(R+1)
        \ge \frac{\varepsilon_1 r_i}{2(R+\varepsilon_1)}.
    \end{align*}
    In particular this holds for $i=t$, giving that $r_t\le \frac{18 R^2}{\varepsilon_1}\cdot \frac{\varepsilon_1}{2(R+\varepsilon_1)}$, which is rather roughly at most $e^R$. Thus $R\ge \ln r_t$ and the desired conclusion follows by Case 1.
    
    Now consider the inductive step for some $i\in [t-1]$.
    Let 
    \begin{align*}
        c 
        \coloneqq R \ln\left(3R/\varepsilon_1\right) + \ln(2R \ln R)
        \le \binom{k_i}{R}. 
    \end{align*}
    It follows from  Lemma~\ref{LEMMA:recursion}, Facts~\ref{FACT:inequalities} and~\ref{FACT:inequalities-b} (note that $\varepsilon_1 \ge 18 R^2/r_t \ge 18 R^2/r_i$), and the inductive hypothesis that 
    \begin{align*}
        \mu(r_{i}+R, r_{i})
        & \le \binom{r_{i}+R}{R} \left(\frac{c}{\binom{k_{i}}{R}} + \frac{\mu(r_{i+1}+R, r_{i+1})}{\mathrm{e}^{c} \cdot \binom{r_{i+1}+R}{R}} \right) \\
        & \le \left(\frac{r_i}{k_i-R}\right)^{R} \cdot c + \left(\frac{r_i}{r_{i+1}}\right)^{R} \cdot \frac{(1+\varepsilon) R \ln R}{\mathrm{e}^{c}} \\
        & \le \left(1+\frac{\varepsilon_1}{R}\right)^{R} \cdot c + \left(\frac{3R}{\varepsilon_1}\right)^{R} \cdot \frac{(1+\varepsilon) R \ln R}{\mathrm{e}^c} \\
        & \le \mathrm{e}^{\varepsilon_1} \cdot c + \mathrm{e}^{R \ln\left(\frac{3R}{\varepsilon_1}\right) - c} \cdot 2 R \ln R \\
        & \le (1+2\varepsilon_1) \left(R \ln\left(\frac{3R}{\varepsilon_1}\right) + \ln(2R \ln R)\right) + 1
        \le (1+\varepsilon) R \ln R, 
    \end{align*}
    as desired. 
    
    This completes the proof of Theorem~\ref{THM:main-mu-r-R-general}. 
\end{proof}

\begin{remark}
We did not  optimise the bound in Theorem~\ref{THM:main-mu-r-R-general}~\ref{THM:main-mu-r-R-general-2} when $R=\Omega(\sqrt r)$, since our main aim was to extend the 
inequality  $\mu(r+R,r)\le (1+o(1)) R\ln R$ for constant $R$ from~\cite{Pik24} 
to as large as possible range of functions~$R(r)$.
\end{remark}
\section*{Acknowledgements}
The authors were supported by ERC Advanced Grant 101020255.
\bibliographystyle{alpha}
\bibliography{TuranSystem}
\end{document}